\pdfoutput=1
\documentclass[colorlinks = true, a4paper, oneside, reqno, 11pt]{amsart}

\usepackage[top = 3.5cm, bottom = 3.5cm, left = 2.5cm, right = 2.5cm]{geometry}
\usepackage[T1]{fontenc}
\usepackage{lmodern}

\usepackage[square,comma,numbers]{natbib}

\usepackage[shortlabels]{enumitem}

\usepackage{mathtools, amssymb, amsthm}

\usepackage{tikz-cd}

\usepackage{doi}

\newtheorem{theorem}{Theorem}
\newtheorem{proposition}[theorem]{Proposition}
\newtheorem{lemma}[theorem]{Lemma}

\theoremstyle{definition} 
\newtheorem{definition}[theorem]{Definition}

\theoremstyle{remark}

\usepackage{hyperref}
\hypersetup{allcolors=[rgb]{0.1,0.1,0.4}}

\newcommand{\Set}{\mathsf{Set}} 
\newcommand{\Cat}{\mathsf{Cat}}
\newcommand{\Cof}{\mathsf{Cof}}
\newcommand{\Lens}{\mathsf{Lens}}
\newcommand{\phibar}{\overline{\varphi}}
\DeclareMathOperator{\cod}{cod}

\begin{document}

\title{Delta lenses as coalgebras for a comonad}

\author{Bryce Clarke}

\address{Centre of Australian Category Theory\\
Macquarie University, NSW 2109, Australia}

\email{bryce.clarke1@hdr.mq.edu.au}

\thanks{The author is supported by the 
Australian Government Research Training Program Scholarship.}





\begin{abstract}
Delta lenses are a kind of morphism between categories which are 
used to model bidirectional transformations between systems. 
Classical state-based lenses, also known as very well-behaved lenses, 
are both algebras for a monad and coalgebras for a comonad. 
Delta lenses generalise state-based lenses, and while delta lenses 
have been characterised as certain algebras for a semi-monad, 
it is natural to ask if they also arise as coalgebras. 

This short paper establishes that delta lenses are coalgebras 
for a comonad, through showing that the forgetful functor 
from the category of delta lenses over a base, to the category 
of cofunctors over a base, is comonadic. 
The proof utilises a diagrammatic approach to delta lenses, 
and clarifies several results in the literature concerning the 
relationship between delta lenses and cofunctors. 
Interestingly, while this work does not generalise the 
corresponding result for state-based lenses, it does provide new 
avenues for exploring lenses as coalgebras. 
\end{abstract}

\maketitle

\section{Introduction}
\label{sec:introduction}

The goal of understanding various kinds of lenses as mathematical 
structures has been an ongoing program in the study of bidirectional 
transformations. 
For example, \emph{very well-behaved lenses} \cite{FGMPS07}, 
also known as \emph{state-based lenses} \cite{AU17},
have been understood as both algebras for a monad \cite{JRW10} and 
coalgebras for a comonad \cite{Con11, GJ12}.
A generalisation of state-based lenses called \emph{category lenses} 
\cite{JRW12} were also introduced as algebras for a monad, 
based on classical work in $2$-category theory on split 
opfibrations \cite{Str74}. 
Another kind of lens between categories called a \emph{delta lens} 
\cite{DXC11} was shown to be a certain algebra for a semi-monad \cite{JR13},
however it remained open as to whether delta lenses could also be 
characterised as (co)algebras for a (co)monad. 

The purpose of this short paper is to characterise delta lenses as 
coalgebras for a comonad (Theorem~\ref{thm:main}). 
The proof of this simple result builds upon and clarifies several recent 
advances in the theory of delta lenses. 

In 2017, Ahman and Uustalu introduced \emph{update-update lenses} 
\cite{AU17} as morphisms of \emph{directed containers} \cite{ACU14}, 
which are equivalent to certain morphisms called
\emph{cofunctors} between categories \cite{Agu97}.
In the same paper, they show explicitly how, using the notation of 
directed containers, delta lenses may be understood as cofunctors with 
additional structure. 

In earlier work \cite{AU16} from 2016, Ahman and Uustalu also provide a 
construction on morphisms of directed containers which yields a 
\emph{split pre-opcleavage} for a functor; in other words, they 
show how cofunctors may be turned into delta lenses. 
We show that this construction is actually a right adjoint to the 
forgetful functor from delta lenses to cofunctors (Lemma~\ref{lemma:right-adjoint}), 
and that the coalgebras for the comonad generated from this adjunction are
delta lenses (Theorem~\ref{thm:main}).   

In 2020, a diagrammatic characterisation of delta lenses was 
introduced by the current author \cite{Cla20}, building upon 
an earlier characterisation of cofunctors as spans \cite{HM93}.
This diagrammatic approach is utilised throughout this paper, 
and leads to another simple characterisation of delta lenses 
(Proposition~\ref{prop:delta-lens}). 

\subsection*{Overview of the paper and related work}

This section provides an informal overview of the paper, together 
with further commentary on the background, and references to related 
work. 
The goal is to provide a conceptual understanding of the results;
later sections will be dedicated to the formal mathematics.

Section~\ref{sec:background} contains the mathematical background 
required for the main results, which are presented in 
Section~\ref{sec:main-result}. 
Consequences of the main result and concluding remarks are in 
Section~\ref{sec:conclusion}.

Throughout the paper we make the assumption that a \emph{system}, 
whatever that may be, can be understood as a category.
The objects of this category are the \emph{states} of the system, 
while the morphisms are the \emph{transitions} (or \emph{deltas}) 
between system states. 

Delta lenses were introduced in \cite[Definition~4]{DXC11} to 
model bidirectional transformations between systems when they are 
understood as categories. 
The \textsc{Get} of a delta lens is a functor 
$f \colon A \rightarrow B$ from the \emph{source category} $A$ to 
the \emph{view category} $B$, while the \textsc{Put} is a certain kind of 
function (that this paper calls a \emph{lifting operation}) 
satisfying axioms analogous to the classical lens laws. 
A slightly modified definition of delta lens appeared in 
\cite[Definition~1]{JR13}, however this definition
still seemed to be ad hoc, and made it difficult to prove deep results
without checking many details.

The definition of delta lens (Definition~\ref{defn:delta-lens}) 
given in this paper is based on a diagrammatic characterisation 
which first appeared in \cite[Corollary~20]{Cla20}, by 
representing the \textsc{Put} in terms of 
bijective-on-objects functors (Definition~\ref{defn:bijective-on-objects})
and discrete opfibrations (Definition~\ref{defn:discrete-opfibration}). 
This diagrammatic approach provides a natural framework for 
studying delta lenses using category theory, and has the benefit of 
allowing for very simple (albeit more abstract) proofs. 
This approach will be utilised throughout this paper, 
although in many places we will also include explicit 
descriptions of constructions using the traditional definition of a 
delta lens. 

A key idea presented in \cite{AU17, Cla20} is that the 
\textsc{Get} and \textsc{Put} of a delta lens can be separated into 
functors and \emph{cofunctors} (Definition~\ref{defn:cofunctor}),
respectively. 
Intuitively, a cofunctor can be understood as a delta lens without 
any information on how the \textsc{Get} acts on morphisms; it is 
the minimum amount of structure needed to specify a \textsc{Put}
operation between categories. 
It was shown in the paper \cite{AU17} that delta lenses are cofunctors 
with additional structure. 
In this paper, we aim to show that said 
structure arises coalgebraically via a comonad. 

Both delta lenses and cofunctors are predominantly understood and 
studied as \emph{morphisms} between categories, however to prove that 
delta lenses are cofunctors equipped with coalgebraic structure, 
it is necessary for them to be understood as \emph{objects}. 
Therefore this paper introduces a new category $\Cof(B)$, whose 
objects are cofunctors into a fixed category $B$ 
(Definition~\ref{defn:category-cofunctors}). 
The category $\Lens(B)$, whose objects are delta lenses into a 
fixed category $B$, was previously studied in \cite{JR17, Cla20b}.
Surprisingly, we show that the category $\Lens(B)$ can be defined 
(Definition~\ref{defn:delta-lens-category}) as the slice category 
$\Cof(B) / 1_{B}$.  
Not only does this provide a new characterisation of delta lenses
in term of cofunctors (Proposition~\ref{prop:delta-lens}), 
but also provides the insight that the canonical forgetful functor 
$L \colon \Lens(B) \rightarrow \Cof(B)$, which takes a delta lens to its 
underlying \textsc{Put} cofunctor, is a projection from a 
slice category. 

Finally, proving that delta lenses are coalgebras for a comonad on 
$\Cof(B)$ amounts to showing that the forgetful functor 
$L \colon \Lens(B) \rightarrow \Cof(B)$ is \emph{comonadic}
(Theorem~\ref{thm:main}). 
A necessary condition is that $L$ has a right adjoint $R$ 
(Lemma~\ref{lemma:right-adjoint}), which 
constructs the \emph{cofree delta lens} from each cofunctor in $\Cof(B)$. 
This construction first appeared explicitly in 
\cite[Section~3.2]{AU16}, however it was not obviously a right adjoint
--- or even a functor --- and it was disconnected from 
the context of cofunctors and delta lenses. 
Both Lemma~\ref{lemma:right-adjoint} and 
Theorem~\ref{thm:main} admit straightforward proofs, with the 
benefit of the diagrammatic approach to cofunctors and delta lenses. 

\subsection*{Notation and conventions}

This section outlines some of the notation and conventions used in the 
paper. 
Given a category~$A$, its underlying set (or discrete category) of 
objects is denoted~$A_{0}$. 
Given a functor $f \colon A \rightarrow B$, its underlying object 
assignment is denoted $f_{0} \colon A_{0} \rightarrow B_{0}$. 
Similarly, a cofunctor $\varphi \colon A \nrightarrow B$ will have 
an underlying object assignment $\varphi_{0} \colon A_{0} \rightarrow B_{0}$. 
Thus the orientation of a cofunctor agrees with the orientation of its
underlying object assignment (this convention is chosen to agree 
with the orientation of delta lenses, however this choice is not uniform 
in the literature on cofunctors). 
The operation $\cod$ sends each morphism to its \emph{codomain} or 
\emph{target} object.

\section{Prerequisites for the main result}
\label{sec:background}

We first recall two special classes of functors, which 
we will use as the building blocks for defining cofunctors and 
delta lenses. 
New contributions in this section include the category $\Cof(B)$ 
whose objects are cofunctors (Definition~\ref{defn:category-cofunctors}), 
and the characterisation of delta lenses as certain morphisms therein
(Proposition~\ref{prop:delta-lens}). 

\begin{definition}\label{defn:bijective-on-objects}
A functor $f \colon A \rightarrow B$ is \emph{bijective-on-objects} if its 
underlying object assignment $f_{0} \colon A_{0} \rightarrow B_{0}$ is a bijection. 
\end{definition}

\begin{definition}\label{defn:discrete-opfibration}
A functor $f \colon A \rightarrow B$ is a \emph{discrete opfibration} if for all pairs,
\[
	(a \in A, \, u \colon fa \rightarrow b \in B)
\]
there exists a unique morphism 
$w \colon a \rightarrow a'$ in $A$ such that $fw = u$. 
\end{definition}

\begin{definition}\label{defn:cofunctor}
A \emph{cofunctor} $\varphi \colon A \nrightarrow B$ between categories is a 
span of functors, 
\begin{equation}
\begin{tikzcd}[column sep = small]
& X 
\arrow[ld, "\varphi"']
\arrow[rd, "\phibar"]
& 
\\
A & & B
\end{tikzcd}
\end{equation}
where $\varphi$ is a bijective-on-objects functor and $\phibar$ is a discrete opfibration. 
\end{definition}

Alternatively, a cofunctor $\varphi \colon A \nrightarrow B$ consists of a
function $\varphi_{0} \colon A_{0} \rightarrow B_{0}$, 
together with a \emph{lifting operation} $\varphi$, which assigns each pair
$(a \in A, \, u \colon \varphi_{0}a \rightarrow b \in B)$ to a morphism 
$\varphi(a, u) \colon a \rightarrow~a'$ in $A$, such that the following axioms are satisfied: 
\begin{enumerate}[(1)]
\item $\varphi_{0}\cod \big( \varphi(a, u) \big) = \cod(u)$;
\item $\varphi(a, 1_{\varphi_{0}a}) = 1_{a}$;
\item $\varphi(a, v \circ u) = \varphi(a', v) \circ \varphi(a, u)$, 
where $a' = \cod\big( \varphi(a, u) \big)$.
\end{enumerate}

\begin{definition}\label{defn:delta-lens}
A \emph{delta lens} $(f, \varphi) \colon A \rightleftharpoons B$ between categories is 
a commutative diagram of functors, 
\begin{equation}\label{eqn:delta-lens}
\begin{tikzcd}[column sep = small]
& X 
\arrow[ld, "\varphi"']
\arrow[rd, "\phibar"]
& 
\\
A 
\arrow[rr, "f"']
& & B
\end{tikzcd}
\end{equation}
where $\varphi$ is a bijective-on-objects functor and $\phibar$ is a discrete opfibration. 
\end{definition}

We can also describe a delta lens $(f, \varphi) \colon A \rightleftharpoons B$ as 
consisting of a functor $f \colon A \rightarrow B$ together with a \emph{lifting operation}
$\varphi$, which assigns each pair
$(a \in A, \, u \colon fa \rightarrow b \in B)$ to a morphism 
$\varphi(a, u) \colon a \rightarrow a'$ in $A$, such that the following axioms are satisfied: 
\begin{enumerate}[(1)]
\item $f \varphi(a, u) = u$;
\item $\varphi(a, 1_{fa}) = 1_{a}$;
\item $\varphi(a, v \circ u) = \varphi(a', v) \circ \varphi(a, u)$,
where $a' = \cod\big( \varphi(a, u) \big)$.
\end{enumerate}

Every delta lens $(f, \varphi) \colon A \rightleftharpoons B$ has an underlying 
functor $f \colon A \rightarrow B$ and an underlying cofunctor $\varphi \colon A \nrightarrow B$,
and their corresponding underlying object assignments are equal; that is, $f_{0} = \varphi_{0}$. 

\begin{definition}\label{defn:category-cofunctors}
For each category $B$, there is a category $\Cof(B)$ of \emph{cofunctors over the base $B$} 
whose objects are cofunctors with codomain $B$, and whose morphisms are given by commutative 
diagrams of functors of the form:
\begin{equation}\label{diagram:morphism-cofunctors}
\begin{tikzcd}[column sep = small]
A 
\arrow[rr, "h"]
& & C
\\
X
\arrow[u, "\varphi"]
\arrow[rr, "\overline{h}"]
\arrow[rd, "\phibar"']
& & Y 
\arrow[u, "\gamma"']
\arrow[ld, "\overline{\gamma}"]
\\
& B &
\end{tikzcd}
\end{equation}
\end{definition}

Equivalently, a morphism in $\Cof(B)$ from a cofunctor $\varphi \colon A \nrightarrow B$ 
to a cofunctor $\gamma \colon C \nrightarrow B$ consists of a functor 
$h \colon A \rightarrow C$ such that $\gamma_{0}ha = \varphi_{0}a$ for all $a \in A$,
and $h\varphi(a, u) = \gamma(ha, u)$ for all pairs 
$(a \in A, \, u \colon \varphi_{0}a \rightarrow b \in B)$.
The functor $\overline{h} \colon X \rightarrow Y$ is then uniquely induced 
from this data. 
Intuitively, if $A$ and $C$ are understood as \emph{source categories} with a fixed
\emph{view category} $B$, then the morphisms in $\Cof(B)$ are functors between the source 
categories which preserve the chosen lifts, given by the 
corresponding cofunctors, from the view category.

\begin{proposition}\label{prop:delta-lens}
Every delta lens $(f, \varphi) \colon A \rightleftharpoons B$ is equivalent to a morphism 
in $\Cof(B)$ whose codomain is the trivial cofunctor on $B$. 
\end{proposition}
\begin{proof}
Consider the morphism in $\Cof(B)$ given by the commutative diagram of functors: 
\begin{equation}\label{diagram:morphism-trivial-cofunctor}
\begin{tikzcd}[column sep = small]
A 
\arrow[rr, "f"]
& & B
\\
X
\arrow[u, "\varphi"]
\arrow[rr, "\phibar"]
\arrow[rd, "\phibar"']
& & B 
\arrow[u, "1_{B}"']
\arrow[ld, "1_{B}"]
\\
& B &
\end{tikzcd}
\end{equation}
The upper commutative square describes a delta lens as given in Definition~\ref{defn:delta-lens}.
Conversely, every delta lens may be depicted as a morphism in $\Cof(B)$ in this way.
\end{proof}

We can unpack \eqref{diagram:morphism-trivial-cofunctor} using the explicit 
characterisation of morphisms in $\Cof(B)$ to obtain the precise difference 
between cofunctors and delta lenses, in terms of objects and morphisms. 
Namely, the diagram \eqref{diagram:morphism-trivial-cofunctor} states that a delta 
lens corresponds to a cofunctor $\varphi \colon A \nrightarrow B$ together with 
a functor $f \colon A \rightarrow B$ such that $fa = \varphi_{0}a$ for all $a \in A$,
and $f\varphi(a, u) = u$ for all pairs $(a \in A, \, u \colon fa \rightarrow b \in B)$.

\begin{definition}
\label{defn:delta-lens-category}
For each category $B$, we define the category of \emph{delta lenses over the base $B$} to be the
slice category $\Lens(B) \coloneqq \Cof(B) \, / \, 1_{B}$, 
where $1_{B}$ is the trivial cofunctor on $B$. 
\end{definition}

By Proposition~\ref{prop:delta-lens}, the objects of $\Lens(B)$ are delta lenses with codomain 
$B$, represented as a morphism into the trivial cofunctor as shown in 
\eqref{diagram:morphism-trivial-cofunctor}. 
The morphisms in $\Lens(B)$ are given by morphisms \eqref{diagram:morphism-cofunctors} 
in $\Cof(B)$ such that the following pasting condition holds: 
\begin{equation}\label{diagram:morphism-delta-lens}
\begin{tikzcd}[column sep = 2em]
A 
\arrow[rr, "h"]
& & C
\arrow[rr, "g"]
& & B
\\
X
\arrow[u, "\varphi"]
\arrow[rr, "\overline{h}"]
\arrow[rrd, "\phibar"']
& & Y 
\arrow[u, "\gamma"']
\arrow[d, "\overline{\gamma}"]
\arrow[rr, "\overline{\gamma}"]
& & B
\arrow[u, "1_{B}"']
\arrow[lld, "1_{B}"]
\\
& & B & & 
\end{tikzcd}
\qquad = \qquad
\begin{tikzcd}[column sep = small]
A 
\arrow[rr, "f"]
& & B
\\
X
\arrow[u, "\varphi"]
\arrow[rr, "\phibar"]
\arrow[rd, "\phibar"'] 
& & B 
\arrow[u, "1_{B}"']
\arrow[ld, "1_{B}"]
\\
& B &
\end{tikzcd}
\end{equation}
In other words, the only additional requirement on a morphism $h \colon A \rightarrow C$ 
between delta lenses over $B$, compared to a morphism between cofunctors over $B$, 
is that $g \circ h = f$. 
This is opposed to just requiring $\gamma_{0}ha = \varphi_{0}a$
on objects (where recall for delta lenses, the underlying object assignments for the 
functor and cofunctor are equal, that is, 
$g_{0} = \gamma_{0}$ and $f_{0} = \varphi_{0}$). 

There is a canonical forgetful functor,
\begin{equation*}
	L \colon \Lens(B) \longrightarrow \Cof(B)
\end{equation*}
which assigns every delta lens to its underlying cofunctor. 
This forgetful functor is the focus of the main result in the following section. 

\section{Main result}
\label{sec:main-result}

While not every cofunctor may be given the structure of a delta lens, 
Ahman and Uustalu \cite{AU16} developed a method which constructs a 
delta lens from any cofunctor.
To understand their construction, first recall that the underlying objects functor 
$(-)_{0} \colon \Cat \rightarrow \Set$ has a right adjoint 
$(\widehat{-}) \colon \Set \rightarrow \Cat$ which takes each set $X$ 
to the \emph{codiscrete category} $\widehat{X}$.

Given a cofunctor $\varphi \colon A \nrightarrow B$ with underlying object assignment 
$\varphi_{0} \colon A_{0} \rightarrow B_{0}$, we may construct the following pullback 
in $\Cat$: 
\begin{equation}
\begin{tikzcd}[row sep = small]
& P
\arrow[ld, "\pi_{A}"']
\arrow[rd, "\pi_{B}"]
\arrow[dd, phantom, "\lrcorner" rotate = -45, very near start]
& \\
A 
\arrow[rd, "\widehat{\varphi}_{0} \, \circ \, \eta_{A}"']
& & B
\arrow[ld, "\eta_{B}"]
\\
& \widehat{B}_{0}
&
\end{tikzcd}
\end{equation} 
Here $\eta_{B} \colon B \rightarrow \widehat{B}_{0}$ is the component of the unit 
for the adjunction at $B$, and $\widehat{\varphi}_{0} \circ \eta_{A}$ the component of 
the unit at $A$ followed by image of $\varphi_{0}$ under the right adjoint. 
Using the universal property of the pullback, we have the following: 
\begin{equation}\label{diagram:right-adjoint}
\begin{tikzcd}[row sep = small]
& X 
\arrow[ldd, bend right, "\varphi"']
\arrow[rdd, bend left, "\phibar"]
\arrow[d, dashed, "{\langle \varphi, \phibar \rangle}"]
& 
\\[+1em]
& P
\arrow[ld, "\pi_{A}"']
\arrow[rd, "\pi_{B}"]
\arrow[dd, phantom, "\lrcorner" rotate = -45, very near start]
& \\
A 
\arrow[rd, "\widehat{\varphi}_{0} \, \circ \, \eta_{A}"']
& & B
\arrow[ld, "\eta_{B}"]
\\
& \widehat{B}_{0}
&
\end{tikzcd}
\end{equation}
Since $\eta_{B}$ is bijective-on-objects, the projection functor $\pi_{A}$ is also
bijective-on-objects which, together with the functor $\varphi$, implies that
$\langle \varphi, \phibar \rangle \colon X \rightarrow P$ is bijective-on-objects,
due to the properties of bijections at the level of objects.
Thus, the upper right triangle in \eqref{diagram:right-adjoint} defines a delta lens 
$P \rightleftharpoons B$. 

The category $P$ has the same objects as $A$, but morphisms $a \rightarrow a'$ in $P$ 
are given by pairs of 
the form $(w \colon a \rightarrow a' \in A, u \colon \varphi_{0}a \rightarrow \varphi_{0}a' \in B)$. 
The functor $\pi_{B} \colon P \rightarrow B$ projects to the second arrow in this 
pair. 
The lifting operation which makes this functor into a delta lens is induced by the 
lifting operation of the original cofunctor; it takes an object $a \in P$ and 
a morphism $u \colon \varphi_{0}a \rightarrow b \in B$ to the morphism 
$\big( \varphi(a, u) \colon a \rightarrow a', u \colon \varphi_{0}a \rightarrow b \big)$ in 
$P$. 

We now show that this construction due to Ahman and Uustalu is universal, in the 
sense that it provides a right adjoint to the functor taking a delta lens to its underlying
cofunctor. 

\begin{lemma}\label{lemma:right-adjoint}
The forgetful functor $L \colon \Lens(B) \rightarrow \Cof(B)$ has a right adjoint.
\end{lemma}
\begin{proof}
Using the construction in \eqref{diagram:right-adjoint}, define the functor 
$R \colon \Cof(B) \rightarrow \Lens(B)$ by the assignment: 
\begin{equation}
\begin{tikzcd}[column sep = small]
& X 
\arrow[ld, "\varphi"']
\arrow[rd, "\phibar"]
& 
\\
A & & B
\end{tikzcd}
\qquad \qquad \longmapsto \qquad \qquad
\begin{tikzcd}[column sep = small]
& X 
\arrow[ld, "{\langle \varphi, \phibar \rangle}"']
\arrow[rd, "\phibar"]
& 
\\
P
\arrow[rr, "\pi_{B}"']
& & B
\end{tikzcd}
\end{equation}
We describe the components of the unit and counit for the adjunction $L \dashv R$ 
and omit the detailed checks that the triangle identities hold. 

Given a cofunctor $\varphi \colon A \nrightarrow B$ the component of the counit is 
given by: 
\begin{equation}
\begin{tikzcd}[column sep = small]
P 
\arrow[rr, "\pi_{A}"]
& & A
\\
X
\arrow[u, "{\langle \varphi, \phibar \rangle}"]
\arrow[rr, equal]
\arrow[rd, "\phibar"']
& & X 
\arrow[u, "\varphi"']
\arrow[ld, "\phibar"]
\\
& B &
\end{tikzcd}
\end{equation}

Given a delta lens $(f, \varphi) \colon A \rightleftharpoons B$ the component of
the unit is given by: 
\begin{equation}
\label{eqn:unit}
\begin{tikzcd}[column sep = 2em]
A 
\arrow[rr, "{\langle 1_{A}, f \rangle}"]
& & P
\arrow[rr, "\pi_{B}"]
& & B
\\
X
\arrow[u, "\varphi"]
\arrow[rr, equal]
\arrow[rrd, "\phibar"']
& & X 
\arrow[u, "{\langle \varphi, \phibar \rangle}"']
\arrow[d, "\phibar"]
\arrow[rr, "\phibar"]
& & B
\arrow[u, "1_{B}"']
\arrow[lld, "1_{B}"]
\\
& & B & & 
\end{tikzcd}
\qquad = \qquad
\begin{tikzcd}[column sep = small]
A 
\arrow[rr, "f"]
& & B
\\
X
\arrow[u, "\varphi"]
\arrow[rr, "\phibar"]
\arrow[rd, "\phibar"'] 
& & B 
\arrow[u, "1_{B}"']
\arrow[ld, "1_{B}"]
\\
& B &
\end{tikzcd}
\end{equation}
The above diagrams show that the pasting condition required in 
\eqref{diagram:morphism-delta-lens} is satisfied. 
\end{proof}

\begin{theorem}\label{thm:main}
The forgetful functor $L \colon \Lens(B) \rightarrow \Cof(B)$ is comonadic. 
\end{theorem}
\begin{proof}
By Lemma~\ref{lemma:right-adjoint}, the functor $L$ has a right adjoint $R$. 
To prove that $L$ is comonadic, it remains to show that the category of coalgebras
for the induced comonad $LR$ on $\Cof(B)$ is equivalent to $\Lens(B)$.

Given a cofunctor $\varphi \colon A \nrightarrow B$, a coalgebra structure map is given
by a morphism in $\Cof(B)$ of the form: 
\begin{equation}\label{eqn:coalgebra}
\begin{tikzcd}[column sep = small]
A 
\arrow[rr, "h"]
& & P
\\
X
\arrow[u, "\varphi"]
\arrow[rr, "\overline{h}"]
\arrow[rd, "\phibar"']
& & X 
\arrow[u, "{\langle \varphi, \phibar \rangle}"']
\arrow[ld, "\phibar"]
\\
& B &
\end{tikzcd}
\end{equation}
However compatibility with the counit forces $\overline{h} = 1_{X}$  
and $h = \langle 1_{A}, f \rangle$, where $f \colon A \rightarrow B$ is a functor 
such that $f \circ \varphi = \phibar$. 
Compatibility with the comultiplication doesn't add any further conditions. 
Therefore, a coalgebra for the comonad $LR$ on $\Cof(B)$ is equivalent to a delta lens
$(f, \varphi) \colon A \rightleftharpoons B$.  
\end{proof}

This theorem establishes the result stated in the title of the paper, 
that delta lenses \eqref{eqn:delta-lens} 
are coalgebras \eqref{eqn:coalgebra} for a comonad. 

\section{Concluding remarks}
\label{sec:conclusion}

In this paper, the category $\Lens(B)$ of delta lenses over the base $B$ 
was characterised as the category of coalgebras for a comonad on 
the category $\Cof(B)$ of cofunctors over the base~$B$. 
This brings together recent results in the study of delta lenses and 
cofunctors. 
In particular, we have shown that the extra structure on cofunctors
given in Ahman and Uustalu's \cite{AU17} characterisation of delta lenses
is coalgebraic, and that their construction of a delta lens from cofunctor 
in the paper \cite{AU16} is precisely the cofree delta lens on a cofunctor. 
Throughout we have also shown how the abstract diagrammatic approach to 
delta lenses, first introduced in \cite{Cla20}, has led to concise 
proofs of these results, and offers a clear perspective on the relationship 
between these ideas. 

Aside from clarification and development of theory, the results presented
in this paper have several other mathematical consequences. 
For example, the functor $L \colon \Lens(B) \rightarrow \Cof(B)$ 
creates all colimits which exist in $\Cof(B)$. 
Thus we can take the coproduct of a pair of cofunctors in $\Cof(B)$, 
and automatically know how to construct the coproduct of the 
corresponding delta lenses in $\Lens(B)$. 

Another consequence from the unit \eqref{eqn:unit} of the adjunction 
between $\Cof(B)$ and $\Lens(B)$ is that every delta lens factorises 
into a bijective-on-objects functor followed by a cofree lens. 
Intuitively, this allows us to first pair every transition in the 
source category $A$ with a transition in the view category $B$ 
via the functor part $f \colon A \rightarrow B$ of the delta lens, 
\[
	w \colon a \rightarrow a' \in A 
	\qquad \longmapsto \qquad 
	(w \colon a \rightarrow a' \in A, fw \colon fa \rightarrow fa' \in B) 
\]
then consider the update propagation determined by the cofunctor 
part $\varphi \colon A \nrightarrow B$ of the delta lens. 
The cofree delta lens on a cofunctor behaves much like an analogue 
of \emph{constant complement} state-based lenses, except that the 
complement is with respect to morphisms rather than objects. 

While the main contributions of this paper are mathematical, it is 
hoped that these results also prompt new ways of understanding delta 
lenses.
For example, previously state-based lenses have been considered 
from a ``\textsc{Put}-based'' perspective \cite{PHF14, FHP15}, 
however this approach could also be adapted to the setting of delta lenses.  
Rather than starting with a \textsc{Get} functor between systems 
and then asking how we might construct a delta lens, 
we might instead start with a \textsc{Put} cofunctor and then 
ask for ways in which this can be given the structure of a delta lens. 
This shift of focus is subtle but important, 
especially in the context of the ideas in \cite{AU17}, 
as it is arguably the \textsc{Put} structure (rather than the \textsc{Get} 
structure) which is central to the study of bidirectional 
transformations and lenses. 

On an separate note, it is worth remarking on the 
similarity between the main result of this paper
and the classical result stating that very well-behaved lenses are 
coalgebras for a comonad \cite{Con11, GJ12}. 
Despite the clear analogy between them, and the inspiration that this paper 
derives from the classical result, it seems that they are unrelated 
at a mathematical level.
The classical result relies on $\Set$ being a cartesian closed 
category, and arises from the adjunction $(-) \times B \dashv [B, - ]$, 
whereas the results in this paper arise from a different adjunction, 
and don't require any aspect of cartesian closure. 

There are many questions to be explored in future work. 
For instance, it is natural to ask if $\Lens(B)$ is comonadic over other 
categories (such as $\Cat$ as was suggested by an anonymous reviewer), 
or if split opfibrations (also known as c-lenses \cite{JRW12}) are also comonadic 
over $\Cof(B)$. 
In recent work by the current author, it has been demonstrated that delta lenses 
arise as algebras for a monad on $\Cat / B$, 
providing a dual to the main result of this paper and strengthening 
the previous work of Johnson and Rosebrugh \cite{JR13}. 
Finally, given the importance of the category $\Lens(B)$ in the study 
of \emph{symmetric lenses} \cite{JR17, Cla20b}, it is also hoped that 
the coalgebraic perspective provides new insights into this area, 
and this will be the subject of further investigation. 

\subsection*{Acknowledgements}

The author would like to thank Michael Johnson for his feedback on this work,
the anonymous reviewers of this paper for their helpful comments, 
and the audience of the Bx2021 workshop for their insightful questions.
The author also thanks Eli Hazel and Giacomo Tendas for their suggestions 
which improved the final version of this paper. 
The author is grateful for the support of the Australian Government Research 
Training Program Scholarship. 

\bibliographystyle{plain}
\bibliography{lenses-as-coalgebras-arXiv.bib}

\begin{thebibliography}{10}

\bibitem{Agu97}
Marcelo Aguiar.
\newblock {\em Internal Categories and Quantum Groups}.
\newblock PhD thesis, Cornell University, August 1997.

\bibitem{ACU14}
Danel Ahman, James Chapman, and Tarmo Uustalu.
\newblock When is a container a comonad?
\newblock {\em Logical Methods in Computer Science}, 10:1--48, 2014.

\bibitem{AU16}
Danel Ahman and Tarmo Uustalu.
\newblock Directed containers as categories.
\newblock In {\em Proceedings 6th Workshop on Mathematically Structured
  Functional Programming}, volume 207 of {\em Electronic Proceedings in
  Theoretical Computer Science}, pages 89--98, 2016.

\bibitem{AU17}
Danel Ahman and Tarmo Uustalu.
\newblock Taking updates seriously.
\newblock In {\em Proceedings of the 6th International Workshop on
  Bidirectional Transformations}, volume 1827 of {\em CEUR Workshop
  Proceedings}, pages 59--73, 2017.

\bibitem{Cla20}
Bryce Clarke.
\newblock Internal lenses as functors and cofunctors.
\newblock In {\em Applied Category Theory 2019}, volume 323 of {\em Electronic
  Proceedings in Theoretical Computer Science}, pages 183--195, 2020.

\bibitem{Cla20b}
Bryce Clarke.
\newblock A diagrammatic approach to symmetric lenses.
\newblock In {\em Applied Category Theory Conference 2020}, volume 333 of {\em
  Electronic Proceedings in Theoretical Computer Science}, pages 79--91, 2021.

\bibitem{DXC11}
Zinovy Diskin, Yingfei Xiong, and Krzysztof Czarnecki.
\newblock From state- to delta-based bidirectional model transformations: the
  asymmetric case.
\newblock {\em Journal of Object Technology}, 10:1--25, 2011.

\bibitem{FHP15}
Sebastian Fischer, ZhenJiang Hu, and Hugo Pacheco.
\newblock The essence of bidirectional programming.
\newblock {\em Science China Information Sciences}, 58:1--21, 2015.

\bibitem{FGMPS07}
J.~Nathan Foster, Michael~B. Greenwald, Jonathan~T. Moore, Benjamin~C. Pierce,
  and Alan Schmitt.
\newblock Combinators for bidirectional tree transformations: A linguistic
  approach to the view-update problem.
\newblock {\em ACM Transactions on Programming Languages and Systems},
  29(3):1--65, 2007.

\bibitem{GJ12}
Jeremy Gibbons and Michael Johnson.
\newblock Relating algebraic and coalgebraic descriptions of lenses.
\newblock In {\em Proceedings of the First International Workshop on
  Bidirectional Transformations}, volume~49 of {\em Electronic Communications
  of the EASST}, pages 1--16, 2012.

\bibitem{HM93}
Philip~J. Higgins and Kirill C.~H. Mackenzie.
\newblock Duality for base-changing morphisms of vector bundles, modules, {L}ie
  algebroids and {P}oisson structures.
\newblock {\em Mathematical Proceedings of the Cambridge Philosophical
  Society}, 114:471--488, 1993.

\bibitem{JR13}
Michael Johnson and Robert Rosebrugh.
\newblock Delta lenses and opfibrations.
\newblock {\em Electronic Communications of the EASST}, 57:1--18, 2013.

\bibitem{JR17}
Michael Johnson and Robert Rosebrugh.
\newblock Universal updates for symmetric lenses.
\newblock In {\em Proceedings of the 6th International Workshop on
  Bidirectional Transformations}, volume 1827 of {\em CEUR Workshop
  Proceedings}, pages 39--53, 2017.

\bibitem{JRW10}
Michael Johnson, Robert Rosebrugh, and R.J. Wood.
\newblock Algebras and update strategies.
\newblock {\em Journal of Universal Computer Science}, 16(5):729--748, 2010.

\bibitem{JRW12}
Michael Johnson, Robert Rosebrugh, and R.J. Wood.
\newblock Lenses, fibrations and universal translations.
\newblock {\em Mathematical Structures in Computer Science}, 22:25--42, 2012.

\bibitem{Con11}
Russell O'Connor.
\newblock Functor is to lens as applicative is to biplate: Introducing
  multiplate, 2011.

\bibitem{PHF14}
Hugo Pacheco, ZhenJiang Hu, and Sebastian Fischer.
\newblock Monadic combinators for putback style bidirectional programming.
\newblock In {\em Proceedings of the ACM SIGPLAN 2014 Workshop on Partial
  Evaluation and Program Manipulation}, volume 333 of {\em PEPM '14}, pages
  39--50, 2014.

\bibitem{Str74}
Ross Street.
\newblock Fibrations and {Y}oneda's lemma in a 2-category.
\newblock In {\em Category Seminar}, volume 420 of {\em Lecture Notes in
  Mathematics}, pages 104--133, 1974.

\end{thebibliography}

\end{document}